\documentclass{amsart}
\usepackage{amsmath}
\usepackage{textcomp}
\usepackage{amsfonts}
\usepackage{amsthm,amscd}
\usepackage{amssymb,enumerate}
\usepackage{amsthm}
\usepackage{caption}
\usepackage{tikz,pgf}
\usepackage{stmaryrd}
\usepackage[all,cmtip]{xy}

\newcommand{\Jac}{\operatorname{Jac}}

\newcommand{\Spec}{\operatorname{Spec}}

\newcommand{\V}{\operatorname{V}}
\newcommand{\Max}{\operatorname{Max}}

\newcommand{\rr}{R\bowtie^f J}
\newcommand{\Nil}{\operatorname{Nil}}

\newcommand{\fn}{\frak{n}}
\newcommand{\fm}{\frak{m}}
\newcommand{\fp}{\frak{p}}

\newcommand{\fq}{\frak{q}}

\newtheorem{thm}{Theorem}[section]
\newtheorem{cor}[thm]{Corollary}
\newtheorem{lem}[thm]{Lemma}
\newtheorem{prop}[thm]{Proposition}
\newtheorem{defn}[thm]{Definition}
\newtheorem{exam}[thm]{Example}
\newtheorem{rem}[thm]{Remark}

\begin{document}
	
	\bibliographystyle{square}

\date{}

\author{Y. Azimi}

\address{Department of Mathematics, University of Tabriz,
	Tabriz, Iran.} \email{u.azimi@tabrizu.ac.ir}

\keywords{Amalgamated algebra, pm-rings, compactly packed rings, trivial extension}

\subjclass[2010]{Primary 13A15, 13B99, 13C05}

\title[spec of the amalgamations]{On the prime spectrum of the amalgamations}

\begin{abstract}
	Let $f:R\to S$ be a ring homomorphism and $J$ be an ideal of $S$. Then the subring $R\bowtie^fJ:=\{(r,f(r)+j)\mid r\in R$ and $j\in J\}$ of $R\times S$ is called the amalgamation of $R$ with $S$ along $J$ with respect to $f$. In this paper, we will deepen the study of the prime spectrum of the amalgamations and characterize when $R\bowtie^fJ$ is a pm-ring. We also investigate the transfer of compactly packed and properly zipped property on amalgamated constructions.
\end{abstract}

\maketitle

\section{Introduction}

The class of \emph{pm-rings} has been studied 
by De Marco and Orsatti in \cite{do}
as rings in which every prime ideal is contained in a unique maximal ideal. The concept has received much attention; 
see for example \cite{m}, \cite{c}, \cite{cc}, \cite{b}, and etc.

The \emph{prime avoidance} lemma is one of the most fundamental results in commutative algebra. In that context,
Reis and Viswanathan in \cite{rv} introduced 
the notion of  \emph{compactly packed} rings
as follows:
A ring $R$ is compactly packed, if
whenever an ideal $I$ of $R$ is contained in the union of a family of
prime ideals of $R$, $I$ is actually contained in one of the prime ideals of the family. 

Also, the set-theoretic dual of prime avoidance, called \emph{prime absorbance}, is well known.
Dual to compactly packed rings, 
the notion of \emph{properly zipped} rings is defined as 
rings in which, whenever a prime ideal contains an intersection of a family of prime ideals, it actually
contains one of them \cite{t}.

Let $R$ and $S$ be two commutative rings with unity,
$J$ be an ideal of $S$ and $f:R\to S$  be a ring homomorphism.
D'Anna, Finocchiaro, and Fontana in \cite{DFF} have introduced the following
subring
$$R\bowtie^fJ:=\{(r,f(r)+j)\mid r\in R\text{ and }j\in J\}$$ of
$R\times S$, called the \emph{amalgamated algebra} (or \emph{amalgamation}) of $R$ with $S$ along $J$
with respect to $f$. This construction generalizes the amalgamated
duplication of a ring along an ideal (introduced and studied in
\cite{DF}). Moreover, several classical constructions such as Nagata's
idealization (cf. \cite[page 2]{Na}), the $R + XS[X]$ and
the $R+XS\llbracket X \rrbracket$ constructions can be studied as
particular cases of this new construction (see \cite[Example 2.5 and Remark 2.8]{DFF}).
Amalgamation, in turn, can be realized as a pullback.
The construction has proved its worth 
providing numerous examples and counterexamples
in commutative ring theory (\cite{DFF1}, \cite{a},
 \cite{ass17}, \cite{f}, \cite{ass}, and etc.).

Prime ideals of $\rr$ are studied in \cite{DFF2}, \cite{DFF1}, \cite{a}, and \cite{ah}.
In this paper, we shall deepen the study of the prime spectrum of the amalgamations, providing some elementary but useful lemmas on
the \emph{inclusions} between ``a prime ideal" and ``unions and
intersections of prime ideals" of $\rr$ (Section 2).
In Section 3, we characterize when the amalgamated algebra is a  pm-ring (Theorem \ref{pmm}). Several corollaries and examples
illustrate various aspects of Theorem \ref{pmm}.
Section 4 is intended to study the transfer of
compactly packed and properly zipped properties.

\section{Preliminaries}
To facilitate the reading of the paper, in this section, we recall notations and
definitions we need later. We also explore inclusions
between a prime ideal of $\rr$ and
intersections and unions of prime ideals.

Let us first fix some notation which we shall use. 
For a commutative ring $A$, the set of
 nilpotent elements,
 prime ideals, 
  and maximal ideals of $A$ 
 will be denoted by 
 $\Nil (A)$,  $\Spec (A)$ and $\Max (A)$  
 respectively. $\V(I)$ denotes the set of
 prime ideals of $A$ containing $I$, and $\Jac (A)$ 
 stands for the Jacobson radical of $A$.

Throughout, we assume that $R$ and $S$ are two commutative rings with unity,
$J$ is a non-zero
proper ideal of $S$ and $f:R\to S$  is a ring homomorphism.
In the sequel, we will use the following
remark, mostly without explicit mention.
\begin{rem}\label{spec}
	\textnormal{(\cite[Proposition 2.6]{DFF2})
		For $\fp\in\Spec(R)$ and $\fq\in\Spec(S)\setminus \V(J)$, set
		\begin{align*}
		\fp^{\prime_f}:= & \fp\bowtie^fJ:=\{(p,f(p)+j)\mid p\in \fp, j\in J\}, \\[1ex]
		\overline{\fq}^f:= & \{(r,f(r)+j)\mid r\in R, j\in J, f(r)+j\in \fq\}.
		\end{align*}
		Then, the following statements hold.
		\begin{itemize}
			\item 
			$\Spec(R\bowtie^fJ)=\{\fp^{\prime_f}\mid \fp\in\Spec(R)\}\cup\{\overline{\fq}^f\mid \fq\in\Spec(S)\setminus V(J)\}$,
			\item $\Max(R\bowtie^fJ)=\{\fp^{\prime_f}\mid \fp\in\Max(R)\}\cup\{\overline{\fq}^f\mid \fq\in\Max(S)\setminus V(J)\}$.
		\end{itemize}
		We call prime (respectively, maximal) ideals of the form
		$\fp^{\prime_f}$ as \emph{type 1} prime (respectively, maximal) ideals, and prime (respectively, maximal) ideals of 
		the form $\overline{\fq}^f$ as \emph{type 2} prime (respectively, maximal) ideals.
				Note that $\rr$ does not have any 
		type 2 prime ideals (respectively, type 2 maximal ideals)
		  if and only if 
		  $J\subseteq \Nil (S)$ (respectively, $J\subseteq \Jac (S)$).
		 }
\end{rem}

The following lemmas, which are of independent interest, have key role in the proof of  results in Section 3 and Section 4.
In some parts of  the  lemmas, we will use the concept of compactly packed property,
which comes in Definition \ref{def}.
Let us start with investigating when a prime ideal
of $\rr$ is contained in a union of prime ideals of $\rr$.

\begin{lem}\label{cp1}
	Let $\fp ,\{\fp_\alpha\}_{\alpha \in \Lambda} \in\Spec(R)$ and
	$\fq, \{\fq_\alpha\}_{\alpha \in \Lambda} \in\Spec(S)\setminus \V(J)$. Then  
	\begin{itemize}
		\item [(1)]
		$ \fp^{\prime_f}\subseteq \cup_{\alpha \in \Lambda}\fp_{\alpha}^{\prime_f}$
		if and only if
		$\fp\subseteq \cup_{\alpha \in \Lambda}\fp_{\alpha}$. 
		\item [(2)]
		If
		$\fq\subseteq \cup_{\alpha \in \Lambda}\fq_{\alpha}$,
		then
		$\ \overline{\fq}^f\subseteq \cup_{\alpha \in \Lambda}   \overline{\fq _{\alpha}}^f$.
		The converse  holds if
		either ``$f$ is surjective" or ``$\Spec(S)\setminus \V(J)$ is compactly packed".
		\item [(3)] 
		The inverse inclusions of parts (1) and (2) also hold (without the additional assumptions for part (2)).
		\item [(4)]
		$ \overline{\fq}^f \subseteq \cup_{\alpha \in \Lambda}\fp_{\alpha}^{\prime_f}$ if and only if $f^{-1}(\fq + J) \subseteq \cup_{\alpha \in \Lambda}\fp_{\alpha}$.
		\item [(5)]
		Let $\Spec(S)\setminus \V(J)$ be compactly packed.
		Then, the following always holds:
		$\fp^{\prime_f}_{\alpha} \nsubseteq \cup_{\alpha \in \Lambda}  \overline{\fq_\alpha}^f$.
	\end{itemize}
\end{lem}
\begin{proof}
	(1) follows immediately
	from the definition of $\fp^{\prime_f}$.\\
	(2) 
	If
	$\fq\subseteq \cup_{\alpha \in \Lambda}\fq_{\alpha}$,
	then it follows,
	from the definition of $\overline{\fq}^f$, that
	$\ \overline{\fq}^f\subseteq \cup_{\alpha \in \Lambda}\overline{\fq_{\alpha}}^f$.
	Conversely,
	assume that $\overline{\fq}^f\subseteq \cup_{\alpha \in \Lambda}\overline{\fq_{\alpha}}^f$.
	Suppose first that $f$ is surjective. Then, for any $f(r)\in\fq$, we have
	$(r,f(r))\in \overline{\fq}^f\subseteq \cup_{\alpha \in \Lambda}\overline{\fq_{\alpha}}^f$.
	Hence $f(r)\in \cup_{\alpha \in \Lambda}\fq_{\alpha}$, as desired. 
	Next, let $\Spec(S)\setminus \V(J)$ be compactly packed.
	Then, $J\nsubseteq \cup_{\alpha \in \Lambda}\fq_{\alpha}$, and one may
	pick  $v\in J\setminus \cup_{\alpha \in \Lambda}\fq_{\alpha}$.
	For any $a\in \fq$,
	we have $(0,av)\in \overline{\fq}^f\subseteq \cup_{\alpha \in \Lambda}\overline{\fq_{\alpha}}^f$,
	which implies
	$av\in  \cup_{\alpha \in \Lambda}\fq_{\alpha}$,
	and so 
	$a\in  \cup_{\alpha \in \Lambda}\fq_{\alpha}$, as desired.
	\\
	(3) 
	The proof is similar to previous parts; omitted.\\
	(4) ($\Rightarrow$): Let $r\in f^{-1}(\fq + J)$, say
	$f(r)= a+i$  for some $a\in q$, $i\in J$.
	Hence $(r,f(r)-i)\in \overline{\fq}^f \subseteq \cup_{\alpha \in \Lambda}\fp_{\alpha}^{\prime_f}$,
	hence that 
	$r\in \cup_{\alpha \in \Lambda}\fp_{\alpha}$.\\
	($\Leftarrow$): Let $(r,f(r)+i)\in \overline{\fq}^f$,
	which implies $f(r)+i =a \in \fq$.
	Then $r\in f^{-1}(\fq + J) \subseteq \cup_{\alpha \in \Lambda}\fp_{\alpha}$, and so 
	$(r,f(r)+i)\in \cup_{\alpha \in \Lambda}\fp_{\alpha}^{\prime_f}$.\\
	(5)
	By assumption,
	we may pick
	$v\in J \setminus \cup_{\alpha \in \Lambda}\fq_{\alpha}$.
	Then,
	$(0,v)\in \fp^{\prime_f}_{\alpha} \setminus \cup_{\alpha \in \Lambda}  \overline{\fq_\alpha}^f$.
\end{proof}

We are now going to study when a prime ideal
of $\rr$ contains an intersection of prime ideals.

\begin{lem}\label{pz1}
	Let $\fp ,\{\fp_\alpha\}_{\alpha \in \Lambda} \in\Spec(R)$ and
	$\fq, \{\fq_\alpha\}_{\alpha \in \Lambda} \in\Spec(S)\setminus \V(J)$.
	Then 
	\begin{itemize}
		\item [(1)]
		$\cap_{\alpha \in \Lambda} \fp_\alpha \subseteq \fp$ if and only if
		$\cap_{\alpha \in \Lambda} \fp_\alpha^{\prime_f} \subseteq \fp^{\prime_f}$.
		\item [(2)]
		$\cap_{\alpha \in \Lambda} \fp_\alpha \supseteq \fp$ if and only if
		$\cap_{\alpha \in \Lambda} \fp_\alpha^{\prime_f}\supseteq \fp^{\prime_f}$.
		\item [(3)]
		$\cap_{\alpha \in \Lambda} \fq_\alpha \subseteq \fq$ if and only if
		$\cap_{\alpha \in \Lambda} \overline{\fq_\alpha}^f \subseteq \overline{\fq}^f$.
		\item [(4)]
		If $\cap_{\alpha \in \Lambda} \fq_\alpha \supseteq \fq$,
		then
		$\cap_{\alpha \in \Lambda} \overline{\fq_\alpha}^f \supseteq \overline{\fq}^f$.
		The converse  holds if
		either ``$f$ is surjective" or ``$\Spec(S)\setminus \V(J)$ is compactly packed".
		\item [(5)]
		$\cap_{\alpha \in \Lambda} \overline{\fq_{\alpha}}
		^f \subseteq \fp^{\prime_f}$ if and only if $f^{-1}(\cap_{\alpha \in \Lambda}\fq_{\alpha} + J) \subseteq \fp$.
		\item [(6)]
		$\cap_{\alpha \in \Lambda} \fp^{\prime_f}_{\alpha} \nsubseteq \overline{\fq}^f$.
		\end{itemize}
\end{lem}
\begin{proof}
	(1) and (2) follow immediately from the
	definition of $\fp^{\prime_f}$ (Remark \ref{spec}).\\
	(3) We prove ``if direction" and the other direction is clear.
    Let
	$y\in \cap_{\alpha \in \Lambda} \fq_\alpha$, and pick
	  $v \in J \setminus \fq$.
	Then $(0,yv)\in \cap_{\alpha \in \Lambda} \overline{\fq_\alpha}^f \subseteq \overline{\fq}^f$, 
	which implies $yv\in \fq$, and so $y\in \fq$.\\
	(4) The proof is similar to the proof of Lemma \ref{cp1}(2),
	and left to the reader.\\
	(5) ($\Rightarrow$): Let $r\in f^{-1}(\cap_{\alpha \in \Lambda}\fq_{\alpha} + J)$, say
	$f(r)= a+i$  for some $a\in \cap_{\alpha \in \Lambda}\fq_{\alpha}$, $i\in J$.
	Hence $(r,f(r)-i)\in \cap_{\alpha \in \Lambda} \overline{\fq_{\alpha}}
	^f \subseteq \fp^{\prime_f}$,
	hence that 
	$r\in \fp$.\\	
	($\Leftarrow$): Let $(r,f(r)+i)\in \cap_{\alpha \in \Lambda} \overline{\fq_{\alpha}}
	^f$,
	which implies $f(r)+i =a\in \cap_{\alpha \in \Lambda}\fq_{\alpha}$.
	Then $r\in f^{-1}(\cap_{\alpha \in \Lambda}\fq_{\alpha} + J) \subseteq \fp$, and so 
	$(r,f(r)+i)\in \fp^{\prime_f}$.
	\\
	(6) For any $v\in J \setminus \fq$,
	$(0,v)\in \cap_{\alpha \in \Lambda} \fp^{\prime_f}_{\alpha} \setminus \overline{\fq}^f$.
\end{proof}

It will be convenient for the reference, if we state now some 
simpler version of the above lemma \cite[Lemma 2.2 and Lemma 2.3]{a}.

\begin{lem}\label{0}
	Let $\fp ,\fp_1, \fp_2\in\Spec(R)$ and
	$\fq, \fq_1 ,\fq_2 \in\Spec(S)\setminus \V(J)$. Then  
	\begin{itemize}
		\item [(1)]
		$\fp_1\subseteq \fp_2$ if and only if $\ \fp_1^{\prime_f}\subseteq \fp_2^{\prime_f}$.
		\item [(2)]
		$\fq_1\subseteq \fq_2$ if and only if
		$\ \overline{\fq_1}^f\subseteq \overline{\fq_2}^f$.
		\item [(3)]
		$ \overline{\fq}^f \subseteq \fp^{\prime_f}$ if and only if $f^{-1}(\fq + J) \subseteq \fp$.
		\item [(4)]
		$\fp^{\prime_f}\nsubseteq \overline{\fq}^f$.
	\end{itemize}
\end{lem}

\section{Transfer of pm condition on amalgamated constructions}

In \cite{do}, the authors studied 
celebrated rings that are called 
\emph{pm-rings}.
The topic is the main subject
of many articles over the years until now.
In this section, we investigate when the
amalgamated algebra $\rr$ is a  pm-ring.
The definition is as follows:

\begin{defn}
	A commutative ring	$R$ is called a \emph{pm-ring}
	if every prime ideal of $R$ is contained in a unique maximal ideal.
\end{defn}

We first give an example that shows the ``if direction'' of
\cite[Theorem 2.4]{pm} is not valid in general. 
Then we correct the mistake of \cite[Theorem 2.4]{pm}.

\begin{exam}\label{ex}
	\textnormal{
		Let $R=k$, $S'=k[x,y]$, 
		$T=S'\setminus \langle x \rangle \cup \langle y \rangle$, and $S=T^{-1}S'$. Let  
		$f:R\to S$ be the composition of the  natural ring
		homomorphisms.
		Let $J =T^{-1} \langle  x\rangle$,
		$\fn =T^{-1} \langle  y\rangle$, and
		$\fq   = 0$.
		Then, the following conditions hold:
		\begin{itemize}
			\item 
			$\Spec (R)= \Max(R) =\{ 0\}$, 
			$\Spec (S) =\{ J, \fn , \fq \}$, and
			$\Max (S) =\{ J, \fn \}$.
			\item 
			$\Spec (\rr)=\{
			\overline{\fq}^f =0,
			\overline{\fn}^f,0^{\prime_f}\}$.
			\item 
			$\Max (\rr) =\{
			\overline{\fn}^f,0^{\prime_f}\}$.
		\end{itemize}
		$R$ is  trivially a pm-ring. Also $S$ has only 
		one maximal ideal not containing $J$.	
		Then, $R$, $S$, and $J$ have conditions of 
		\cite[Theorem 2.4]{pm}, but $\rr$
		 is not a  pm-ring, since $\overline{\fq}^f=0$
		is a prime ideal that contained in both
		$\overline{\fn}^f$ and $0^{\prime_f}$.
		}
\end{exam}

\begin{rem}
	Example \ref{ex} also shows that \cite[Corollary 4.3]{pm} is not valid in general. 
	 One can easily check that the assumptions of
	  ``if direction'' of the corollary hold, but
	 $\rr$ is not an \emph{h-local} ring, since 
	 it is not a pm-ring.
\end{rem}

We now characterize when the
amalgamated algebra $\rr$ is a  pm-ring.
This is  one of the main results of the paper.

\begin{thm}\label{pmm}
	$\rr$
	is a pm-ring if and only if $R$
	is a pm-ring and, for any
	$\fq \in \Spec(S)\setminus \V(J)$,
	the following equality holds:
	\begin{equation*}\label{dag}
		\tag*{(\S)} \big\lvert \big(\Max (S) \cap \V(\fq) \big) \setminus \V(J)\big\rvert + 
		\big\lvert \Max (R) \cap \V(f^{-1}(\fq +J)) \big\rvert =1. 
	\end{equation*}
\end{thm}
\begin{proof} 
($\Rightarrow$): Assume that $\rr$
is a pm-ring. Then $R$ is a pm-ring since it is  homomorphic image of $\rr$.
Now let $\fq\in\Spec(S)\setminus \V(J)$.
By Remark \ref{spec} and Lemma \ref{0}, the set of maximal ideals of $\rr$
that contains the prime ideal
$\overline{\fq}^f$, is the union of disjoint sets
$\{\overline{\fn}^f \mid \fn \in \big(\Max (S) \cap \V(\fq) \big) \setminus \V(J) \}$ and
$\{ \fm^{\prime_f}\mid \fm \in \Max (R) \cap \V(f^{-1}(\fq +J) \}$.
Therefore, the assumption implies the desired conclusion:
$$
\big\lvert \big(\Max (S) \cap \V(\fq) \big) \setminus \V(J)\big\rvert + 
\big\lvert \Max (R) \cap \V(f^{-1}(\fq +J)) \big\rvert =
\big\lvert \Max (\rr) \cap \V(\overline{\fq}^f) \big\rvert =1.$$

($\Leftarrow$):
Assume that $R$
is a pm-ring and, for any
$\fq \in \Spec(S)\setminus \V(J)$, the equality
\ref{dag} holds.
We show that, for any $\mathcal{P} \in\Spec(\rr)$,
$\lvert \Max (\rr) \cap \V (\mathcal{P}) \rvert =1$.
Two cases arise:\\
\textbf{Case 1.}
$\mathcal{P}=\fp^{\prime_f}$ with
$\fp \in \Spec (R)$.
As $R$ is a pm-ring, there exists a unique 
maximal ideal $\fm$ of $R$ that contains $\fp$.
On the other hand, by Lemma \ref{0}, 
$\fp^{\prime_f}$ is not contained
in any type 2 maximal ideal of $\rr$. Hence
$\fm^{\prime_f}$ is the unique 
maximal ideal of $\rr$ that contains $\fp^{\prime_f}$.\\
\textbf{Case 2.}
$\mathcal{P}=\overline{\fq}^f$ with
$\fq \in \Spec(S)\setminus \V(J)$.
In the left side of the equality \ref{dag},
one term is equal to $1$ and the other term is equal to $0$.
Suppose first that 
the second term of \ref{dag} is $0$, and
the first term  is $1$, with $\fn$ as the only element of $\big(\Max (S) \cap \V(\fq) \big) \setminus \V(J)$. 
Then, by Lemma \ref{0}(3), 
there is no type 1
maximal ideal of $\rr$ that contains $\overline{\fq}^f$;
while, by Lemma \ref{0}(2),
 $\overline{\fn}^f$ is the unique type 2
 maximal ideal of $\rr$ that contains $\overline{\fq}^f$.
 Thus $\lvert \Max (\rr) \cap \V (\overline{\fq}^f) \rvert =1$.\\
Next suppose that
the first term of \ref{dag} is $0$, and the second is $1$, with $\fm$ as the only element of $\Max (R) \cap \V(f^{-1}(\fq +J)$.
Then, by Lemma \ref{0}(2), 
there is no type 2
maximal ideal of $\rr$ that contains $\overline{\fq}^f$;
while, by Lemma \ref{0}(3), $\fm^{\prime_f}$ is the unique
maximal ideal of $\rr$ that contains $\overline{\fq}^f$.
Again we showed that 
$\lvert \Max (\rr) \cap \V (\overline{\fq}^f) \rvert =1$.
This completes the proof.
\end{proof}

We now consider some special cases of amalgamations 
to be pm-rings.
Recall that if $f:=id_R$ is the identity homomorphism
on $R$, and $I$ is an ideal of $R$, 
then $R\bowtie I:=R\bowtie^{id_R} I$ is called the amalgamated duplication of $R$ along $I$.

\begin{cor}\label{pmdup}
$R\bowtie I$
is a pm-ring if and only if so is $R$.
\end{cor}
\begin{proof}
	By Theorem \ref{pmm}, $R\bowtie I$
	is a pm-ring if and only if $R$
	is a pm-ring and, for any
	$\fq \in \Spec(R)\setminus \V(I)$,
	the following equality holds:
	\begin{equation*}
		\big\lvert \big(\Max (R) \cap \V(\fq) \big) \setminus \V(I)\big\rvert + 
		\big\lvert \Max (R) \cap \V(\fq +I) \big\rvert =1. 
	\end{equation*}	
	Note that the left side of the above equality,  counts the number
	of ``maximal ideals that  contain $\fq$ but not $I$" plus
	``maximal ideals that contain both $\fq$ and $I$";
	i.e., the number of all maximal ideals that contain $\fq$.
	Therefore when we assume that $R$
	is a pm-ring, the equality above has not 
	further hypothesis.
\end{proof}

Note that non-local domains, trivially are never
pm-rings. Note also that $R\bowtie I$ is never
a domain. Corollary \ref{pmdup} shows a way of
constructing rings that are not pm-rings, from 
trivial ones (non-local domains).

The next corollary deals with \emph{trivial extensions}.
Let $M$ be an $R$-module. 
Then $R\ltimes M$, the trivial
	extension of $R$ by $M$, is a special case of $\rr$
with $J^2=0$ (see \cite[Remark 2.8]{DFF}). 

\begin{cor}\label{pmi}
	 The 	following statements hold:
	\begin{itemize}
		\item [(1)]
		Let $J\subseteq \Jac (S)$. Then
		$R\bowtie^f J$ is a pm-ring if and only if so is  $R$.
		\item [(2)]
		Let $M$ be an $R$-module. Then
		$R\ltimes M$ is a pm-ring if and only if so is  $R$.
	\end{itemize}
\end{cor}
\begin{proof}
	 (2) is
	a special case of (1). In order to prove (1),
let $J\subseteq \Jac (S)$ and $R$ be a pm-ring.	
Let $\fq \in \Spec(S)\setminus \V(J)$.
We show that the equality \ref{dag}
in the proof of
the Theorem \ref{pmm} holds.
 The first term of the equality is $0$, 
 since $J\subseteq \Jac (S)$.
The second term is  $1$
since every maximal ideal of $R$ that contains
$ f^{-1}(\fq +J)$, also contains
 the prime ideal $ f^{-1}(\fq )$ of $R$.
\end{proof}

In the proof of Theorem \ref{pmm}, we noticed that
$\rr$ is a pm-ring if and only if
one of the first and second term of \ref{dag} is equal to
$0$ and the other is equal to $1$.  
Example \ref{ex} draws conditions in which
both terms of \ref{dag} are equal to $1$, and $\rr$ is not a pm-ring.
In the  Corollary \ref{pmi}, the first term is $0$ and 
the second term is $1$.
We now provide a situation that the second term is $0$.
Here we will use the trivial
 fact that zero-dimensional rings are pm.
 
\begin{exam}
\textnormal{	Let $R$ be a zero-dimensional ring which is not local.
	Let $f:R\to R$ be the identity map and
	 $I$ be a maximal ideal of $R$. 
	 For any $\fq \in \Spec(R)\setminus \V(I)$,
	 $\fq +I=R$. Thus, here, the second term of \ref{dag} is equal to $0$.
	 On the other hand, $\fq$ itself is the unique maximal ideal
	 that contains $\fq$, and so, 
	 the first term of \ref{dag} is equal to $1$.
	 Note that, by Corollary \ref{pmdup}, $R\bowtie  I$ is a  pm-ring.}
\end{exam}


\section{Compactly packed and properly zipped property}

In this section
we investigate
the transfer of compactly packed and properly zipped property on the amalgamations.
For more information and useful references about 
these properties, we
refer the reader to \cite{t}.
Let us first  introduce the concept of compactly packed condition
for an arbitrary subset of $\Spec (R)$.

\begin{defn}\label{def}
	Let $X \subseteq \Spec (R)$. We say that $X$ is \emph{compactly packed} if
	whenever an ideal $I$ of $R$ is contained in the union of a family $\{\fp_i \}_i$ of
	elements of $X$, then $I\subseteq \fp_i$, for some $i$. We simply say that $R$ is compactly packed if the whole $\Spec (R)$ is compactly packed.
\end{defn}

Note  that, in the definition,
we can assume that $I$ is prime (use \cite[Theorem 1]{kap}).
We shall use this fact frequently without any comment.

\begin{thm}\label{cp}
	If $\rr$ is  compactly packed, then 
	$R$ and $\Spec(S)\setminus \V(J)$ are compactly packed.
\end{thm}
\begin{proof}
	 Assume that $\rr$ is  compactly packed.
	Then $R$ is compactly packed since it is  homomorphic image of $\rr$.
	To prove $\Spec(S)\setminus \V(J)$ is compactly packed,
	let $\{\fq_\alpha\}_{\alpha \in \Lambda} \in\Spec(S)\setminus \V(J)$, and  $\fq\in \Spec(S)$ such that
	$\fq\subseteq \cup_{\alpha \in \Lambda}\fq_{\alpha}$.
	We need to prove that, for some $\alpha \in \Lambda$,
	$\fq\subseteq \fq_{\alpha}$.
	
	We claim that $J\nsubseteq \fq$.
	Once the claim proved, we shall have
	$\overline{\fq}^f \in \Spec (\rr)$, 
	and by Lemma \ref{cp1}(2), $\overline{\fq}^f\subseteq \cup_{\alpha \in \Lambda}\overline{\fq_{\alpha}}^f$.
	Hence, by assumption, 
	$\overline{\fq}^f \subseteq \overline{\fq_{\alpha}}^f$,
	for some $\alpha \in \Lambda$.
	Thus
	$\fq\subseteq \fq_{\alpha}$, by Lemma \ref{0}.\\
	It remains to prove the claim.	
	Suppose on the contrary that 
	$J\subseteq \fq$, and set
	$\fp:=f^{-1}(\fq) \in \Spec (R)$.
	Then,  for any
	$(r, f(r)+i)\in \fp^{\prime_f}$, 
	one has 
	$f(r)+i \in \fq \subseteq \cup_{\alpha \in \Lambda}\fq_{\alpha}$
	which implies
	$(r, f(r)+i)\in \cup_{\alpha \in \Lambda}\overline{\fq_{\alpha}}^f$. 
	Therefore,
	$\fp^{\prime_f}\subseteq \cup_{\alpha \in \Lambda}\overline{\fq_{\alpha}}^f$.
	The assumption now leads to
	$\fp^{\prime_f}\subseteq \overline{\fq_{\alpha}}^f$,
	for some $\alpha \in \Lambda$.
	This contradicts Lemma \ref{0}.
\end{proof}

It is not easy to prove the converse of
Theorem \ref{cp}.
But we take some steps in this direction.

\begin{prop}
	Let $R$ be compactly packed. Then
	$\{\fp^{\prime_f}_{\alpha}\}_{\alpha \in \Lambda}$
	is compactly packed, for any family
	$\{\fp_{\alpha}\}_{\alpha \in \Lambda}\in \Spec(R)$.
\end{prop}
\begin{proof}
Let $\mathcal{P}, \{\fp^{\prime_f}_{\alpha}\}_{\alpha \in \Lambda} \in\Spec(\rr)$ and 
$\mathcal{P} \subseteq \cup_{\alpha \in \Lambda}\fp^{\prime_f}_{\alpha}$.
We show that, for some $\alpha \in \Lambda$,
$\mathcal{P}\subseteq \fp^{\prime_f}_{\alpha}$.
The argument  splits into two cases. Suppose first that
$\mathcal{P}=\overline{\fq}^f$ with
$\fq \in \Spec(S)\setminus \V(J)$.
By Lemma \ref{cp1}(4), we have $f^{-1}(\fq + J) \subseteq \cup_{\alpha \in \Lambda}\fp_{\alpha}$.
The assumption now gives 
$f^{-1}(\fq + J) \subseteq \fp_{\alpha}$, for some
$\alpha \in \Lambda$.	
Another use of Lemma \ref{cp1}
will yield the desired conclusion that
$\overline{\fq}^f \subseteq \fp^{\prime_f}_{\alpha}$,
as desired.
Argument on the case $\mathcal{P}=\fp^{\prime_f}$ with
$\fp \in \Spec (R)$ is similar, and omitted.
\end{proof}

Recall that, if $J\subseteq \Nil (S)$, then $\rr$ does not have any 
type 2 prime ideals.
Thus we have the following result.

\begin{cor} \label{cmf}
	The 	following statements hold:
	\begin{itemize}
		\item [(1)]
		Let $J\subseteq \Nil (S)$. Then
		$R\bowtie^f J$ is compactly packed if and only if so is  $R$.
		\item [(2)]
		Let $M$ be an $R$-module. Then
		$R\ltimes M$ is compactly packed if and only if so is  $R$.
	\end{itemize}
\end{cor}

\begin{prop}
    Let $\Spec(S)\setminus \V(J)$ be compactly packed. Then
	$\{\overline{\fq}^f_{\delta}\}_{\delta \in \Delta}$
	is compactly packed, for any family
	$\{\fq_{\delta}\}_{\delta \in \Delta}\in \Spec(S)\setminus \V(J)$.
\end{prop}
\begin{proof}
	 By Lemma \ref{cp1}(5), the only case we need
	 to consider is the case
	 $\overline{\fq}^f\subseteq \cup_{\delta \in \Delta}\overline{\fq}^f_\delta$ with
	 $\fq \in \Spec(S)\setminus \V(J)$. Then, by Lemma \ref{cp1}(2),
	$\fq\subseteq \cup_{\delta \in \Delta}\fq_\delta$, and
	our assumption implies
	$\fq\subseteq \fq_\delta$, for some $\delta \in \Delta$.
	Thus 
	$\overline{\fq}^f\subseteq\overline{\fq}^f_\delta$, as desired. 
\end{proof}

The dual notion of  compactly packed rings are
properly zipped rings.
In the following, as our final result, we characterize when
$\rr$ is properly zipped.
The definition is as follows.
\begin{defn}
	We say that a ring $R$ is a \emph{properly zipped} ring if whenever a prime ideal $\fp$ of $R$ contains the intersection of a family $\{\fp_i\}_i$ of prime ideals of $R$, then $\fp_i \subseteq \fp$,  for some $i$.	
\end{defn}

\begin{thm}\label{pz}
	Let $f:R\to S$ be surjective. Then $\rr$ is  a properly zipped ring
	if and only if so is $R$.
\end{thm}
\begin{proof}
	($\Rightarrow$): If $\rr$ is a properly zipped ring, then, so is $R$ since it is  homomorphic image of $\rr$.
	
	($\Leftarrow$):
	Assume that 
	$R$ is properly zipped.
	Let $\mathcal{P}$, $\{\mathcal{P}_\alpha\}_{\alpha \in \Lambda} \in\Spec(\rr)$ and 
	$\cap_{\alpha \in \Lambda}\mathcal{P}_{\alpha} \subseteq \mathcal{P}$.
		We need to show that
		$\mathcal{P}_{\alpha} \subseteq \mathcal{P}$, for some $\alpha \in \Lambda$.
	The argument  splits into two cases:
	
	\textbf{Case 1.} $\mathcal{P}=\fp^{\prime_f}$ with
	$\fp \in \Spec (R)$.
	Arrange $\cap_{\alpha \in \Lambda}\mathcal{P}_{\alpha}$
	such that
	$ (\cap_{\delta \in \Delta}\overline{\fq}^f_\delta) \cap (\cap_{\gamma \in \Gamma} \fp^{\prime_f}_{\gamma}) \subseteq \fp^{\prime_f}$, with
	$\fq_\delta \in \Spec(S)\setminus \V(J)$ and 
	$\fp_\gamma  \in \Spec (R)$.
	Then, either 
	$\cap_{\gamma \in \Gamma} \fp^{\prime_f}_{\gamma} \subseteq \fp^{\prime_f}$
	or
	$ \cap_{\delta \in \Delta}\overline{\fq}^f_\delta\subseteq \fp^{\prime_f}$.
	If the first inclusion occurs, then one easily drives the 
	desired conclusion ($\fp^{\prime_f}_{i} \subseteq \fp^{\prime_f}$) by double use of Lemma \ref{pz1}(1)
	together with the assumption.\\ 
	Now let the second inclusion occur. Then, by Lemma \ref{pz1}(5),
	$f^{-1}(\cap_{\delta \in \Delta}\fq_{\delta} + J) \subseteq \fp$,
	and so
	$\cap_{\delta \in \Delta} f^{-1}(\fq_{\delta}) +f^{-1} (J) \subseteq \fp$.
	Applying the assumption on the inclusion
	``$\cap_{\delta \in \Delta} f^{-1}(\fq_{\delta}) \subseteq \fp$" yields
	$f^{-1}(\fq_{\delta}) \subseteq \fp$, for some 
	$\delta \in \Delta$.
	Therefore, since $f$ is surjective,
	$f^{-1}(\fq_{\delta}+J) = f^{-1}(\fq_{\delta}) +f^{-1} (J) \subseteq \fp$.
	We	now use Lemma \ref{0}(3) to have
	$\overline{\fq}^f_\delta\subseteq \fp^{\prime_f}$, which completes the proof.
	
	\textbf{Case 2.}
	$\mathcal{P}=\overline{\fq}^f$ with
	$\fq \in \Spec(S)\setminus \V(J)$.
	Let 
	$ (\cap_{\delta \in \Delta}\overline{\fq}^f_\delta) \cap (\cap_{\gamma \in \Gamma} \fp^{\prime_f}_{\gamma}) \subseteq \overline{\fq}^f$, with
	$\fq_\delta \in \Spec(S)\setminus \V(J)$ and 
	$\fp_\gamma  \in \Spec (R)$. 
	By Lemma \ref{pz1}(6), 
	$\cap_{\gamma \in \Gamma} \fp^{\prime_f}_{\gamma} \nsubseteq \overline{\fq}^f$,
	and so
	$ \cap_{\delta \in \Delta}\overline{\fq}^f_\delta\subseteq \overline{\fq}^f$.
	By Lemma \ref{pz1}(3)  we have
	$ \cap_{\delta \in \Delta}\fq_\delta\subseteq \fq$.
	Note that $S$ is a properly zipped ring, since
	$f$ is surjective, and so
	$ \fq_\delta\subseteq \fq$, for some $\delta \in \Delta$.
	Lemma \ref{pz1}(3) now gives the desired inclusion
	$\overline{\fq}^f_\delta\subseteq \overline{\fq}^f$.
\end{proof}

Let us mention two  consequences of the theorem.
The first one is direct result of the theorem,
while the second is obtained from the proof of it.
\begin{cor}\label{dup}
	Let I be an ideal of R. Then $R\bowtie I$ is a properly zipped ring if and only if so is $R$.
\end{cor}

\begin{cor} \label{pzi}
		Let $J\subseteq \Nil (S)$. Then
		$R\bowtie^f J$ is properly zipped if and only if so is  $R$.
		In particular, for an $R$-module $M$,
		$R\ltimes M$ is properly zipped if and only if so is  $R$.
\end{cor}
\begin{proof}
	$\rr$ does not have any 
	type 2 prime ideals, and the proof is
	just as the first paragraph of the proof of ``Case 1 of Theorem \ref{pz}".
\end{proof}

The next example shows that, in Theorem \ref{pz} the surjective hypothesis
for $f$ is not superfluous.

\begin{exam}
	Let  $\mathbb{Q}$ be the field of rational numbers, $\mathbb{R}$ be the
	field of real numbers and $X$ is an indeterminate over $\mathbb{R}$.
	Let $f:\mathbb{Q}\to \mathbb{R}[X]$ be the natural embedding
	and $J:=X\mathbb{R}[X]$.
	Then $\mathbb{Q} \bowtie^f J$ is isomorphic to
	$\mathbb{Q}+X \mathbb{R}[X]$ (see \cite[Example 2.5]{DFF}).\\
	Clearly $\mathbb{Q}$ is a properly zipped ring.
	But $\mathbb{Q}+X \mathbb{R}[X]$ is not properly zipped,
	since it has infinitely many maximal ideals \cite[Corollary 4.4]{t}.
\end{exam}

\end{document}